\documentclass[11pt]{article}
\usepackage{amsfonts}
\usepackage{latexsym, amssymb, amsmath, amscd, amsfonts, epsfig, graphicx, colordvi}
\usepackage{ifpdf}
\usepackage{graphicx}
\usepackage{xcolor}
\usepackage{tikz}
\usepackage{caption} 
\usetikzlibrary{arrows.meta}

\usepackage[dvipsnames]{pstricks}
\newtheorem{thm}{Theorem}[section]

\newtheorem{pro}[thm]{Proposition}
\newtheorem{rem}[thm]{\it Remarks}
\newtheorem{defi}[thm]{Definition}
\newtheorem{lem}[thm]{Lemma}

\newtheorem{exam}[thm]{\bf Example}
\def\pf{\noindent{\it Proof.} }
\def\qed{\nopagebreak\hfill{\rule{4pt}{7pt}}
\medbreak}

\def\qed{\nopagebreak\hfill{\rule{4pt}{7pt}}
\medbreak}

\makeatletter
\def\ExtendSymbol#1#2#3#4#5{\ext@arrow 0099{\arrowfill@#1#2#3}{#4}{#5}}
\makeatother

\title{An Andrews--Gordon Type Identity Related to Andrews' Parity Consideration}
\author{
Robert X.J. Hao\raisebox{5pt}{\scriptsize 1},
Xiaorui Niu\raisebox{5pt}{\scriptsize 2},
Doris D.M. Sang\raisebox{5pt}{\scriptsize 3}
and
Diane Y.H. Shi\raisebox{5pt}{\scriptsize 4}
}

\date{
\vspace{15pt}
\raisebox{5pt}{\scriptsize 1\,}haoxj@njit.edu.cn\\
College of Science and Mathematics\\
Nanjing Institute of Technology, Nanjing, Jiangsu 211167, P.~R.~China\\[10pt]
\raisebox{5pt}{\scriptsize 2\,}xiaorui\_niu@tju.edu.cn\\
School of Mathematics\\
Tianjin University, Tianjin 300072, P.~R.~China\\[10pt]
\raisebox{5pt}{\scriptsize 3\,}sangdm@dufe.edu.cn\\
School of Mathematics\\
Dongbei University of Finance and Economics, Dalian, Liaoning 116025, P.~R.~China\\[10pt]
\raisebox{5pt}{\scriptsize 4\,}shiyahui@tju.edu.cn\\
School of Mathematics\\
Tianjin University, Tianjin 300072, P.~R.~China\\[10pt]
}

\begin{document}
\maketitle
\noindent {\bf Abstract.}
Andrews investigated parity conditions in the Rogers–Ramanujan–Gordon theorem. 
Under the conditions that even parts or odd parts appear an even number of times, 
Andrews discovered two Rogers–Ramanujan–Gordon type partition theorems  
and derived corresponding generating functions. In the Rogers-Ramanujan-Gordon 
theorem, there are two parameters $k$ and $a$, where $k-1$ is the maximum 
number of consecutive parts $l$ and $l+1$, and $a-1$ is the maximum number of 
parts equal to $1$. Andrews' first theorem deals with the case 
$k\equiv a \;(\rm{mod}\;2)$, while the second theorem concerns the case 
where $k$ is even and $a$ is odd. These two partition identities have different 
infinite product forms on the right-hand side. In this paper, we consider the case 
$k\not\equiv a \;(\rm{mod}\;2)$ and obtain an Andrews–Gordon 
type identity whose right-hand side coincides with that of Andrews' identity for the
case $k\equiv a \;(\rm{mod}\;2)$. By fixing the number of peaks of the corresponding 
lattice paths, we also derive a recurrence system whose solution agrees with the 
product-side generating function. We were unable to find a suitable combinatorial 
interpretation of the infinite sum form of this expression in terms of partitions, 
but with the help of lattice paths, we provide an appropriate combinatorial interpretation.
Finally, we prove this identity analytically by applying Bailey's lemma.



\section{ Introduction}

In 2010, Andrews \cite{and10} considered parity conditions in the Rogers--Ramanujan--Gordon theorem and obtained 
two new Rogers--Ramanujan--Gordon type theorems, along with their corresponding generating function forms of the Andrews--Gordon type.

Let us recall the significant Rogers--Ramanujan-Gordon theorem due to Gordon \cite{gor61} which is a combinatorial generalization of the Rogers--Ramanujan identities.

\begin{thm}\label{GRR}
For $1\leq a \leq k$, let $B_{k,a}(n)$ denote the number of partitions
of $n$ of the form $(1^{f_1}, 2^{f_2}, 3^{f_3}, \ldots)$, where $f_i$ is the
number of occurrences of $i$ in the partition, such that
$f_1\leq a-1$ and
 $f_i+f_{i+1}\leq k-1$.
 Let $A_{k,a}(n)$ denote the number of partitions of $n$ whose  parts  are not congruent to
$0,\pm a$ modulo $2k+1$. Then for all $n\geq 0$,

\[A_{k,a}(n)=B_{k,a}(n).\]
\end{thm}

Below is a brief review of some basic definitions related to partitions.  
An unrestricted partition $\lambda$ of a positive integer $n$ is a 
non-increasing sequence of positive integers 
$\lambda_1\geq \cdots\geq \lambda_s>0$ 
such that $n=|\lambda|=\lambda_1+\cdots+\lambda_s$. The unique partition of $0$ is the empty partition $\emptyset$.
Given a partition $\lambda$, let $f_l(\lambda)$  denote the number of occurrences of
$l$ in $\lambda$. Then a partition $\lambda$ can also be represented as $\lambda=(1^{f_1},2^{f_2},3^{f_3},\ldots)$. 

 Andrews \cite{and74} discovered the following Andrews--Gordon identity, which serves as the generating function version of the Rogers--Ramanujan--Gordon theorem.

\begin{thm}{\rm (\cite[Theorem 1]{and74})}\label{AG}
For $1\leq a\leq k$, we have
\begin{equation}\label{eqRRG}
\sum_{n_1, n_2, \cdots, n_{k-1}\geq 0}\frac{q^{N_1^2+N_2^2+\cdots+N_{k-1}^2+N_a+\cdots+N_{k-1}}}
{(q)_{n_1}(q)_{n_2}\ldots(q)_{n_{k-1}}}=
\frac{(q^a,q^{2k+1-a},q^{2k+1};q^{2k+1})_{\infty}}
{(q)_{\infty}},
\end{equation}
where \(N_j=n_j+n_{j+1}+\cdots+n_{k-1}\) for \(1\le j\le k-1\), \(N_k=0\), and equivalently \(n_i=N_i-N_{i+1}\) for \(1\le i\le k-1\). We shall use this notation throughout the paper.
\end{thm}

Here and throughout  this paper, we employ the standard notation on $q$-series:
\[
(a)_\infty=(a;q)_{\infty}=\prod_{j=0}^{\infty}(1-aq^j), \ \
(a)_n=(a;q)_n=\frac{(a)_{\infty}}{(aq^n)_{\infty}}
\]
and
\[
(a_1,\ldots,a_k;q)_{\infty}
=(a_1;q)_{\infty}\cdots(a_k;q)_{\infty}.
\]

The product side of \eqref{eqRRG} is immediately seen to be the
generating function for the congruence classes counted by $A_{k,a}(n)$.
The interpretation of the sum side as the generating function for
the difference conditions counted by $B_{k,a}(n)$ is more subtle.
It was first obtained analytically by Andrews via recurrences, and a
bijective/lattice-path proof was later given by Warnaar
\cite{war97}. Warnaar's construction may be viewed in terms of particle
motion on frequency diagrams, starting from a minimal configuration
satisfying the Gordon frequency conditions.

Kur\c{s}ung\"{o}z subsequently used Gordon marking to obtain sum-side
formulas in several parity-restricted cases \cite{kur09}; in the cases
relevant here, Gordon marking is closely related to Warnaar's particle
motion approach.

Later, Andrews \cite{and10} studied Rogers--Ramanujan--Gordon type partitions with parity restrictions and obtained infinite product generating functions in certain cases, 
together with the corresponding Andrews--Gordon type identities. 
The remaining cases were later investigated by Kur\c{s}ung\"{o}z \cite{kur09} and Kim and Yee \cite{kim13}. 
The results are summarized as follows.

\begin{thm}[Andrews \cite{and10}, Kim and Yee \cite{kim13}, Kur\c{s}ung\"{o}z \cite{kur09}]\label{andwe}
Suppose that $k\ge a\ge 1$ are integers. Let $W_{k,a}(n)$ denote the number of partitions enumerated by $B_{k,a}(n)$ with the additional 
restriction that even parts appear an even number of times.

If $k\equiv a \pmod{2}$, then
\begin{align}\nonumber
\sum_{n\geq 0}W_{k,a}(n)q^n
&=\frac{(-q;q^2)_\infty(q^a,q^{2k+2-a},q^{2k+2};q^{2k+2})_\infty}{(q^2;q^2)_\infty}\\
&=\label{andw}\sum_{N_1\geq N_2\geq\cdots\geq N_{k-1}\geq0}
\frac{q^{N_1^2+N_2^2+\cdots+N_{k-1}^2+2N_a+2N_{a+2}+\cdots+2N_{k-2}}}
{(q^2;q^2)_{N_1-N_2}\cdots(q^2;q^2)_{N_{k-2}-N_{k-1}}(q^2;q^2)_{N_{k-1}}}.
\end{align}

If $k\not\equiv a \pmod{2}$, then
\begin{align}
\sum_{n\geq0}W_{k,a}(n)q^n
&=\sum_{N_1\geq N_2\geq\cdots\geq N_{k-1}\geq0}
\frac{q^{N_1^2+N_2^2+\cdots+N_{k-1}^2+2N_a+2N_{a+2}+\cdots+2N_{k-1}}}
{(q^2;q^2)_{N_1-N_2}\cdots(q^2;q^2)_{N_{k-2}-N_{k-1}}(q^2;q^2)_{N_{k-1}}}\label{Wkad-sum}\\
&=\frac{(-q^3;q^2)_\infty(q^{a+1},q^{2k+1-a},q^{2k+2};q^{2k+2})_\infty}{(q^2;q^2)_\infty}\label{Wkad-prod1}\\
&\qquad+\frac{q(-q^3;q^2)_\infty(q^{a-1},q^{2k+3-a},q^{2k+2};q^{2k+2})_\infty}{(q^2;q^2)_\infty}.\label{Wkad-prod2}
\end{align}
\end{thm}

Here, \eqref{andw} is due to Andrews \cite{and10}. 
In the case $k\not\equiv a \pmod{2}$, the multiple-sum identity \eqref{Wkad-sum} was derived by Kur\c{s}ung\"{o}z \cite{kur09} using a combinatorial method, 
while the product formula \eqref{Wkad-prod1}--\eqref{Wkad-prod2} was derived by Kim and Yee \cite{kim13}.

\begin{thm}[Andrews \cite{and10}, Kur\c{s}ung\"{o}z \cite{kur09}, Kim and Yee \cite{kim13}]\label{andoverw}
Let $\overline{W}_{k,a}(n)$ denote the number of partitions enumerated by $B_{k,a}(n)$ with the additional restriction that odd parts appear an even number of times.

If $k\ge a\ge 1$ with $k$ odd and $a$ even, then
\begin{align}\nonumber
\sum_{n\geq0}\overline{W}_{k,a}(n)q^n
&=\frac{(-q^2;q^2)_\infty(q^a,q^{2k+2-a},q^{2k+2};q^{2k+2})_\infty}{(q^2;q^2)_\infty}\\
&=\label{andwbar}\sum_{N_1\geq N_2\geq\cdots\geq N_{k-1}\geq0}
\frac{q^{N_1^2+N_2^2+\cdots+N_{k-1}^2+n_1+n_3+\cdots+n_{a-3}+N_{a-1}+N_a+\cdots+N_{k-1}}}
{(q^2;q^2)_{n_1}\cdots(q^2;q^2)_{n_{k-2}}(q^2;q^2)_{n_{k-1}}}.
\end{align}

If $k\ge a\ge 1$ with $k$ even and $a$ odd, then
\begin{align}
\sum_{n\geq 0}\overline{W}_{k,a}(n)q^n
&=\label{OWakd}\sum_{N_1\geq N_2\geq\cdots\geq N_{k-1}\geq0}
\frac{q^{N_1^2+N_2^2+\cdots+N_{k-1}^2+n_1+n_3+\cdots+n_{a-2}+N_a+\cdots+N_{k-1}}}
{(q^2;q^2)_{n_1}\cdots(q^2;q^2)_{n_{k-2}}(q^2;q^2)_{n_{k-1}}}\\
&=\label{OWakd2}\frac{(-q^2;q^2)_\infty(q^{a+1},q^{2k+1-a},q^{2k+2};q^{2k+2})_\infty}{(q^2;q^2)_\infty}.
\end{align}
\end{thm}

Here, \eqref{andwbar} is due to Andrews \cite{and10}. 
When $k$ is even and $a$ is odd, the multiple-sum identity \eqref{OWakd} was derived by Kur\c{s}ung\"{o}z \cite{kur09} using a combinatorial method; the same identity was also obtained independently by Kim and Yee \cite{kim13}, who further derived the product formula \eqref{OWakd2}.

Thus, the infinite product and multisum generating functions for $W_{k,a}(n)$ and $\overline{W}_{k,a}(n)$ have been determined in all cases. 
For $W_{k,a}(n)$, the identity \eqref{andw} corresponds to the case $k\equiv a \pmod{2}$, whereas the opposite-parity case $k\not\equiv a \pmod{2}$ is described by the multisum formula \eqref{Wkad-sum} together with the product formula \eqref{Wkad-prod1}--\eqref{Wkad-prod2}. 
Unlike \eqref{andw}, whose right-hand side is a single infinite product, the corresponding product side in the case $k\not\equiv a \pmod{2}$ is expressed as a sum of two infinite products. 
Therefore, our objective is to derive, for the case $k\not\equiv a \pmod{2}$, an Andrews--Gordon type identity whose right-hand side coincides with that of \eqref{andw}. 
We obtain the following theorem, which will be proved later by using Bailey pairs and Bailey's lemma.

\begin{thm}\label{main1}
For $k\geq a\geq 1$ and $k\not\equiv a\;(\rm{mod}\;2)$,
\begin{align}\label{sg}
\nonumber\sum_{n_1, n_2, \cdots, n_{k-1}\geq 0}&
\frac{q^{N_1^2+N_2^2+\cdots+N_{k-1}^2+2N_a+2N_{a+2}+\cdots+2N_{k-1}}(-q;q^2)_{n_{k-1}}}
{(q^2;q^2)_{n_1}(q^2;q^2)_{n_2}\cdots(q^2;q^2)_{n_{k-2}}(q^4;q^4)_{n_{k-1}}}
\\&=\frac{(-q;q^2)_{\infty}(q^a,q^{2k+2-a},q^{2k+2};q^{2k+2})_{\infty}}{(q^2;q^2)_{\infty}}.
\end{align}
\end{thm}

It can be seen that the left-hand side of \eqref{sg} is a multisum form function, 
which differs from Andrews' multisum generating function form. Therefore, it cannot be directly interpreted as the partition number $W_{k,a}(n)$. 
For this reason, we aim to provide a combinatorial interpretation of it. 

Since we have not been able to find a suitable explanation based on Rogers--Ramanujan--Gordon type partitions, we will instead use lattice paths with three unit steps to 
offer an appropriate combinatorial interpretation. 
This lattice path approach goes back to Burge's work on recursive correspondences \cite{bur82}, which was reformulated by Andrews and Bressoud in terms of binary words and lattice paths \cite{andbre85}, 
and later recast by Bressoud in terms of paths with three unit steps \cite{bre89}.

We consider lattice paths of finite length lying in the first quadrant. These paths start on the $y$-axis and terminate on the $x$-axis. 
Each step belongs to one of the following three types:
\begin{itemize}
	\item North-East $NE$: $(x,y)\to (x+1,y+1)$,
	\item South-East $SE$: $(x,y)\to (x+1,y-1)$,
	\item East (Horizontal) $E$: $(x,0)\to(x+1,0)$.
\end{itemize}

The path is either empty or ends with a $SE$ step.
East steps are allowed only along the $x$-axis.

A peak is a vertex $(x,y)$ that is preceded by a $NE$ step and followed by a $SE$ step. 
The height of a vertex is its $y$-coordinate, and its weight is its $x$-coordinate. 
The major index of a path $\mathcal P$, denoted by $\operatorname{maj}(\mathcal P)$, is the sum of the weights of all peaks of $\mathcal P$. For the empty path, we set $\operatorname{maj}(\mathcal P)=0$.

We shall also use the notion of relative height of a peak. This notion was introduced by Bressoud for his lattice path model \cite{bre89}, and later reformulated in a simpler form by Berkovich and Paule \cite{berpau01}. 

The relative height of a peak $(x,y)$ is the largest integer $h$ for which we can find two vertices on the path, $(x',y-h)$ and $(x'',y-h)$, 
such that $x'<x<x''$ and such that between these two vertices there are no peaks of height greater than $y$ and every peak of height $y$ has abscissa at least $x$.

We now state the following theorem.

\begin{thm}\label{main}
For $k\geq a\geq 1$ with $k\not\equiv a\pmod 2$, let $S_{k,a}(n)$ be the number of lattice paths $\mathcal P$ of major index $n$ satisfying the following conditions:\begin{enumerate}
\item $\mathcal P$ starts at $(0,k+1-a)$ and remains at height at most $k$;
\item $E$-step sondition: the number of $E$ steps along the segment of $\mathcal P$ from the starting point to each peak of relative height $k$ or $k-1$, and along the segment between any two such peaks, is a multiple of $4$;
\item Parity condition: for every peak of weight $x$ and relative height $r$, one has
$x\equiv r\pmod 2$.

\end{enumerate}
Then
\begin{equation}\label{s}
\sum_{n=0}^{\infty}S_{k,a}(n)q^n
=
\sum_{n_1,\ldots,n_{k-1}\geq 0}
\frac{
q^{N_1^2+\cdots+N_{k-1}^2+2N_a+2N_{a+2}+\cdots+2N_{k-1}}
(-q;q^2)_{n_{k-1}}
}
{
(q^2;q^2)_{n_1}\cdots(q^2;q^2)_{n_{k-2}}
(q^4;q^4)_{n_{k-1}}
},
\end{equation}
where $ 1\leq j\leq k-1$.

\end{thm}

The remainder of this paper is organized as follows. In Section 2, we fix the number of peaks in the lattice paths and derive recurrence relations for the corresponding generating functions, from which the product formula follows. In Section 3, we give a direct lattice path interpretation of the multisum in \eqref{s}, which proves Theorem \ref{main}. In Section 4, we recall Bailey's lemma and the Bailey-pair transformations needed for the analytic proof. Finally, in Section 5, we use these transformations to prove Theorem \ref{main1}, namely identity \eqref{sg}.

\section{Path recurrences and product formula}\label{sec:fixedpeak}

In this section, we prove the product formula for the path family appearing in Theorem \ref{main}. More precisely, we show that the product side of \eqref{sg} is the generating function for these paths. The direct construction of the paths from the multisum side will be given in the next section. Here, we work only with the path conditions stated in Theorem \ref{main}.

The proof is based on a refinement according to the number of peaks. We fix the number $N$ of peaks and, for each admissible starting height, define a generating function $L_{k,i}(N)$. These functions satisfy a finite recurrence system with respect to  the index $i$. On the product side, after applying Jacobi's triple product and a summation formula to \eqref{sg}, we obtain functions $\mathcal U_{k,i}(N)$ satisfying the same recurrence system and the same initial values. A uniqueness argument then identifies $L_{k,i}(N)$ with $\mathcal U_{k,i}(N)$, and summing over $N$ gives the desired product formula.

This recurrence method is in the spirit of classical recurrence proofs of Andrews--Gordon type identities, but here the recurrence is organized by the number of peaks. It is also close to Mallet's treatment of generalized Bressoud--Burge paths; see \cite[Sec.~II.5]{mallet08}. The summation formula used below to expand the product side is obtained from \cite[Lem.~II.14]{mallet08} by a simple specialization and limit.

Throughout this section, let $k\geq 1$ be fixed, and assume that all indices $i$ satisfy $1\leq i\leq k+1$ and $i\equiv k+1\pmod 2$.

For such an index $i$ and for $N\geq 0$, let $\mathcal L_{k,i}(N)$ denote the set of lattice paths satisfying the parity condition and the $E$-step condition in Theorem \ref{main}, starting at height $k+1-i$, remaining at height at most $k$, and having exactly $N$ peaks. Define
\[
L_{k,i}(N)=\sum_{\mathcal P\in\mathcal L_{k,i}(N)}q^{\operatorname{maj}(\mathcal P)}.
\]
 We use the convention $L_{k,i}(N)=0$ for $N<0$, and $L_{k,i}(0)=1$.

Since $k\not\equiv a\pmod 2$, the path generating function in Theorem \ref{main} is $\sum_{N\geq0}L_{k,a}(N)$. Thus the product formula to be proved in this section is
\begin{equation}\label{eq:path-product}
\sum_{n\geq0}S_{k,a}(n)q^n=\sum_{N\geq0}L_{k,a}(N)=
\frac{(-q;q^2)_\infty(q^a,q^{2k+2-a},q^{2k+2};q^{2k+2})_\infty}{(q^2;q^2)_\infty}.
\end{equation}

\medskip
\noindent\textbf{A parity observation.}
Before deriving the recurrences, we record a simple observation explaining how the parity condition on peaks is reflected in the local structure of the paths.

Call a vertex an \emph{upturn} if it is followed by a $NE$ step and satisfies one of the following conditions: it is the initial vertex, it is preceded by a $SE$ step, or it lies on the $x$-axis and is preceded by an $E$ step. Thus an upturn is either an ordinary valley, or a point at which a horizontal step on the $x$-axis is followed by an upward step.

Let $P$ be a peak of relative height $r$, and let $V=(x_0,y_0)$ be the left endpoint of the horizontal line segment at height $h(P)-r$ which realizes the relative height of $P$. Then $V$ is an upturn. If the smaller branches between $V$ and $P$ are deleted, the remaining outer part from $V$ to $P$ consists of $r$ upward steps. Each deleted branch starts and ends at the same height, and hence has even horizontal length. Therefore the deletion changes the $x$-coordinate of $P$ only by an even integer, and we have
\[
x(P)\equiv x_0+r\pmod 2,
\]
where $x(P)$ denotes the weight of the peak $P$. Since $r$ is the relative height of $P$, the condition
$x(P)\equiv r\pmod 2$ is equivalent to $x_0\equiv0\pmod 2.$
Thus, for the paths considered here, the parity condition on peaks is equivalent to requiring that every upturn have even weight. 

With this observation in hand, we now decompose paths according to their leftmost local configuration. This gives the following recurrence system for the fixed-peak generating functions.

\begin{pro}\label{pro:path-recurrence}
For $N\geq1$, the functions $L_{k,i}(N)$ satisfy the following recurrence system. The first equation depends on the parity of $k$:
\begin{align}
L_{k,1}(N)&=q^{2N}L_{k,3}(N),
&&\text{if $k$ is even}, \label{eq:L-top-even}\\
L_{k,2}(N)&=q^{2N}L_{k,4}(N)+q^{2N-1}L_{k,2}(N-1),
&&\text{if $k$ is odd}. \label{eq:L-top-odd}
\end{align}
For every index $i$ such that
$3\leq i\leq k-1, i\equiv k+1\pmod 2,$
we have
\begin{equation}\label{eq:L-internal}
(1+q^{4N})L_{k,i}(N)
=q^{2N}\bigl(L_{k,i-2}(N)+L_{k,i+2}(N)\bigr)
+(q^{4N-2}+q^{2N-1})L_{k,i}(N-1).
\end{equation}
Finally, for $i=k+1$, we have
\begin{equation}\label{eq:L-bottom}
(1+q^{4N})L_{k,k+1}(N)
=2q^{2N}L_{k,k-1}(N)
+(q^{4N-2}+q^{2N-1})L_{k,k+1}(N-1).
\end{equation}
\end{pro}

\pf
We use the parity observation above. In particular, an upturn cannot occur at weight $1$. Hence a path cannot begin with $SENE$, and, at the bottom, an initial $E$ step cannot be followed immediately by a $NE$ step.

Suppose first that $k$ is even. Then the relevant highest index is $i=1$, and the path starts at height $k$. Since the height is at most $k$, the first step cannot be $NE$. Since an $E$ step is allowed only at height $0$, the first step cannot be $E$ either. Thus the first step must be $SE$. The second step cannot be $NE$, for otherwise the path would begin with $SENE$ and would have an upturn of odd weight. Hence every such path begins with two $SE$ steps. Removing these two steps gives a path counted by $L_{k,3}(N)$. Conversely, adjoining two initial $SE$ steps to a path counted by $L_{k,3}(N)$ shifts all its $N$ peaks two units to the right. Therefore
\[
L_{k,1}(N)=q^{2N}L_{k,3}(N).
\]

Suppose next that $k$ is odd. Then the relevant highest index is $i=2$, and the path starts at height $k-1$. If the path begins by going down, the first two steps must be $SESE$, since $SENE$ is forbidden by the parity observation. Removing these two steps gives a path counted by $L_{k,4}(N)$, and adjoining them back shifts all $N$ peaks two units to the right. This gives the contribution
\[
q^{2N}L_{k,4}(N).
\]
The other possibility is that the path begins with a peak of relative height $1$, namely with the local pattern $NESE$. Removing this peak gives a path counted by $L_{k,2}(N-1)$. Conversely, adjoining $NESE$ to a path counted by $L_{k,2}(N-1)$ creates a new peak of weight $1$ and shifts the other $N-1$ peaks two units to the right. Hence the increase in the major index is
\[
1+2(N-1)=2N-1,
\]
and this gives the contribution
\[
q^{2N-1}L_{k,2}(N-1).
\]
Combining the two cases gives
\[
L_{k,2}(N)=q^{2N}L_{k,4}(N)+q^{2N-1}L_{k,2}(N-1).
\]

Now let
\[
3\leq i\leq k-1,\qquad i\equiv k+1\pmod 2.
\]
We decompose the paths according to their leftmost local configuration.

First consider paths obtained by adjoining two initial $NE$ steps to a path counted by $L_{k,i-2}(N)$. This operation shifts all $N$ peaks two units to the right, and hence gives the term
\[
q^{2N}L_{k,i-2}(N).
\]
However, if the path counted by $L_{k,i-2}(N)$ itself begins with two $SE$ steps, then after adjoining the two $NE$ steps the initial part becomes $NENESESE$. This is not to be counted in this case. The overcounted paths in $L_{k,i-2}(N)$ are precisely those obtained by adjoining two initial $SE$ steps to paths counted by $L_{k,i}(N)$, and this contributes $q^{2N}L_{k,i}(N)$ inside $L_{k,i-2}(N)$. After the outer factor $q^{2N}$, the overcount is
\[
q^{4N}L_{k,i}(N).
\]
Thus this first case contributes
\[
q^{2N}L_{k,i-2}(N)-q^{4N}L_{k,i}(N).
\]

Second, if the path begins by going down, then the first two steps must be $SESE$, since $SENE$ is forbidden. Removing these two steps gives a path counted by $L_{k,i+2}(N)$. Conversely, adjoining two initial $SE$ steps shifts all $N$ peaks two units to the right, and gives the contribution
\[
q^{2N}L_{k,i+2}(N).
\]

Third, the path may begin with a peak of relative height $1$, namely with the local pattern $NESE$. Removing this peak gives a path counted by $L_{k,i}(N-1)$. Conversely, adjoining $NESE$ creates a new peak of weight $1$ and shifts the remaining $N-1$ peaks two units to the right. Hence the contribution is
\[
q^{1+2(N-1)}L_{k,i}(N-1)=q^{2N-1}L_{k,i}(N-1).
\]

Fourth, the path may begin with a peak of relative height $2$, namely with the local pattern $NENESESE$. Removing this peak gives a path counted by $L_{k,i}(N-1)$. Conversely, adjoining $NENESESE$ creates a new peak of weight $2$ and shifts the remaining $N-1$ peaks four units to the right. Hence the contribution is
\[
q^{2+4(N-1)}L_{k,i}(N-1)=q^{4N-2}L_{k,i}(N-1).
\]

Adding these four contributions, we obtain
\[
L_{k,i}(N)
=
q^{2N}L_{k,i-2}(N)-q^{4N}L_{k,i}(N)
+q^{2N}L_{k,i+2}(N)
+q^{2N-1}L_{k,i}(N-1)
+q^{4N-2}L_{k,i}(N-1).
\]
Moving the overcounting term to the left-hand side gives
\[
(1+q^{4N})L_{k,i}(N)
=
q^{2N}\bigl(L_{k,i-2}(N)+L_{k,i+2}(N)\bigr)
+
(q^{4N-2}+q^{2N-1})L_{k,i}(N-1).
\]

It remains to consider $i=k+1$. In this case the path starts at height $0$. Thus it cannot begin with a $SE$ step. The cases in which the path begins with $NESE$ or $NENESESE$ are exactly as above and give
\[
(q^{4N-2}+q^{2N-1})L_{k,k+1}(N-1).
\]
The case in which the path begins with two $NE$ steps gives
\[
q^{2N}L_{k,k-1}(N)-q^{4N}L_{k,k+1}(N),
\]
by the same overcounting argument used above.

There is one additional possibility at the bottom: the path may begin with an $E$ step. Since $N\geq1$, the path eventually leaves the $x$-axis. By the parity observation, the first upward step after the initial horizontal part must occur at an even weight; in particular, the path begins with at least two consecutive $E$ steps. For such paths, remove the first two $E$ steps. Then raise by two units the portion of the path from immediately after these two steps up to the next occurrence of two consecutive $E$ steps on the $x$-axis, and replace that next pair of $E$ steps by $SESE$. This gives a path counted by $L_{k,k-1}(N)$. The transformation preserves the number of peaks and changes the major index by subtracting $2$ from each of the $N$ peak weights. Hence this class contributes
\[
q^{2N}L_{k,k-1}(N).
\]
The construction is reversible, so no path is lost or counted twice.

Combining the three contributions for $i=k+1$, we get
\[
L_{k,k+1}(N)
=
q^{2N}L_{k,k-1}(N)-q^{4N}L_{k,k+1}(N)
+
(q^{4N-2}+q^{2N-1})L_{k,k+1}(N-1)
+
q^{2N}L_{k,k-1}(N).
\]
After moving the overcounting term to the left-hand side, this becomes
\[
(1+q^{4N})L_{k,k+1}(N)
=
2q^{2N}L_{k,k-1}(N)
+
(q^{4N-2}+q^{2N-1})L_{k,k+1}(N-1).
\]

It remains to check that the above operations preserve the condition on $E$ steps in Theorem \ref{main}.  The initial blocks
\[
NENE,\qquad SESE,\qquad NESE,\qquad NENESESE
\]
contain no $E$ steps. Hence, for every old peak of relative height $k$ or $k-1$, the number of $E$ steps from the starting point to that peak, and the number of $E$ steps between any two such peaks, are unchanged. If a new leftmost peak is relevant in the exceptional small cases, then no $E$ step precedes it, and its $E$-step distance to any old relevant peak is the same as the old distance from the starting point to that peak.

For the operation involving bottom $E$ steps, the only change in the number of bottom $E$ steps before the later part of the path is by a multiple of $4$, since the construction uses two bottom blocks $EE$. Peaks inside the shifted segment have no $E$ steps between them. Thus the required multiple-of-$4$ condition is preserved in all cases.

This completes the proof.
\qed

We next record a finite summation needed to expand the product side. It is obtained by a simple specialization and limiting case of Mallet's summation formula \cite[Lem.~II.14]{mallet08}.

\begin{lem}\label{lem:finite-transform}
For every integer $r$,
\begin{equation}\label{eq:finite-transform}
\sum_{N\geq |r|}
q^{N^2}
\frac{(-q;q^2)_N}
{(q^2;q^2)_{N-r}(q^2;q^2)_{N+r}}
=
q^{r^2}\frac{(-q;q^2)_\infty}{(q^2;q^2)_\infty}.
\end{equation}
\end{lem}

\pf
This identity is a specialization of \cite[Lem.~II.14]{mallet08}. We include the short derivation for completeness. In Mallet's notation, where $(z)_m=(z;q)_m$, the lemma states that, for every integer $n$,
\[
\sum_{N\geq |n|}
\frac{
(-q^n/a,-q^n/b)_{N-n}(abq)^{N-n}(-aq,-bq)_n
}
{
(q)_{N+n}(q)_{N-n}
}
=
\frac{(-aq,-bq)_\infty}{(q,abq)_\infty}.
\]
We now make the substitutions
\[
q\mapsto q^2,\qquad a\mapsto q^{-1},\qquad n\mapsto r,
\]
and then let $b\to0$.

It is enough to consider $r\geq0$, since both sides of \eqref{eq:finite-transform} are invariant under $r\mapsto -r$. After the above substitutions, the left-hand side of Mallet's identity becomes
\[
\sum_{N\geq r}
\frac{
(-q^{2r+1};q^2)_{N-r}
(-q^{2r}/b;q^2)_{N-r}
(bq)^{N-r}
(-q,-bq^2;q^2)_r
}
{
(q^2;q^2)_{N+r}(q^2;q^2)_{N-r}
}.
\]
Put $m=N-r$. Then
\[
\lim_{b\to0}
(-q^{2r}/b;q^2)_m(bq)^m
=
q^{m^2+2rm}
=
q^{N^2-r^2}.
\]
Moreover,
\[
(-q^{2r+1};q^2)_{N-r}(-q;q^2)_r
=
(-q;q^2)_N,
\]
and
\[
\lim_{b\to0}(-bq^2;q^2)_r=1.
\]
Thus the left-hand side tends to
\[
\sum_{N\geq r}
q^{N^2-r^2}
\frac{(-q;q^2)_N}
{(q^2;q^2)_{N-r}(q^2;q^2)_{N+r}}.
\]

On the right-hand side, the same substitutions give
\[
\frac{(-q,-bq^2;q^2)_\infty}{(q^2,bq;q^2)_\infty},
\]
and hence, as $b\to0$,
\[
\frac{(-q,-bq^2;q^2)_\infty}{(q^2,bq;q^2)_\infty}
\longrightarrow
\frac{(-q;q^2)_\infty}{(q^2;q^2)_\infty}.
\]
Therefore
\[
\sum_{N\geq r}
q^{N^2-r^2}
\frac{(-q;q^2)_N}
{(q^2;q^2)_{N-r}(q^2;q^2)_{N+r}}
=
\frac{(-q;q^2)_\infty}{(q^2;q^2)_\infty}.
\]
Multiplying both sides by $q^{r^2}$ proves \eqref{eq:finite-transform} for $r\geq0$. The case $r<0$ follows by the symmetry $r\mapsto -r$.
\qed

By Jacobi's triple product identity,
\begin{equation}\label{eq:jtp-i}
(q^i,q^{2k+2-i},q^{2k+2};q^{2k+2})_\infty
=\sum_{r\in\mathbb Z}(-1)^r q^{(k+1)r^2+(k+1-i)r}.
\end{equation}
Combining \eqref{eq:finite-transform} with \eqref{eq:jtp-i}, we obtain
\begin{equation}\label{eq:product-side-expansion}
\frac{(-q;q^2)_\infty(q^i,q^{2k+2-i},q^{2k+2};q^{2k+2})_\infty}{(q^2;q^2)_\infty}
=
\sum_{N\geq0}q^{N^2}(-q;q^2)_N
\sum_{r=-N}^{N}
\frac{(-1)^r q^{kr^2+(k+1-i)r}}
{(q^2;q^2)_{N-r}(q^2;q^2)_{N+r}}.
\end{equation}
This motivates the definition
\begin{equation}\label{eq:U-def}
\mathcal U_{k,i}(N)
=q^{N^2}(-q;q^2)_N
\sum_{r=-N}^{N}(-1)^r
\frac{q^{kr^2+(k+1-i)r}}
{(q^2;q^2)_{N-r}(q^2;q^2)_{N+r}}.
\end{equation}
Thus \eqref{eq:product-side-expansion} says that the product side is $\sum_{N\geq0}\mathcal U_{k,i}(N)$.

\begin{pro}\label{pro:U-recurrence}
For $N\geq1$, the functions $\mathcal U_{k,i}(N)$ satisfy the same recurrence system as the functions $L_{k,i}(N)$ in Proposition \ref{pro:path-recurrence}, with the same initial values
$\mathcal U_{k,i}(0)=1$.

\end{pro}

\pf
The initial value is immediate from \eqref{eq:U-def}.  Put
\[
b=k+1-i
\]
and denote the summand in \eqref{eq:U-def} by
\[
U_b(N,r)
=
q^{N^2}(-q;q^2)_N(-1)^r
\frac{q^{kr^2+br}}
{(q^2;q^2)_{N-r}(q^2;q^2)_{N+r}}.
\]
Thus
\[
\mathcal U_{k,i}(N)=\sum_{r=-N}^{N}U_b(N,r).
\]
Since replacing $i$ by $i-2$ increases $b$ by $2$, while replacing $i$ by $i+2$ decreases $b$ by $2$, we have
\[
\mathcal U_{k,i-2}(N)=\sum_{r=-N}^{N}q^{2r}U_b(N,r),
\qquad
\mathcal U_{k,i+2}(N)=\sum_{r=-N}^{N}q^{-2r}U_b(N,r).
\]

For $|r|\leq N-1$, a direct calculation gives
\[
\frac{U_b(N-1,r)}{U_b(N,r)}
=
q^{-2N+1}
\frac{(1-q^{2N-2r})(1-q^{2N+2r})}
{1+q^{2N-1}}.
\]
Equivalently,
\begin{equation}\label{eq:ratio-use}
(q^{4N-2}+q^{2N-1})U_b(N-1,r)
=
(1-q^{2N-2r})(1-q^{2N+2r})U_b(N,r).
\end{equation}

We first verify \eqref{eq:L-internal} with $L$ replaced by $\mathcal U$.  Subtracting the right-hand side of \eqref{eq:L-internal} from the left-hand side and collecting the coefficient of a fixed $U_b(N,r)$, we obtain
\[
1+q^{4N}
-q^{2N+2r}
-q^{2N-2r}
-(1-q^{2N-2r})(1-q^{2N+2r}).
\]
This expression is identically zero.  The endpoint terms $r=\pm N$ cause no additional contribution, since the factor
\[
(1-q^{2N-2r})(1-q^{2N+2r})
\]
vanishes at the corresponding endpoint.  Hence \eqref{eq:L-internal} holds for $\mathcal U_{k,i}(N)$.

Next consider the equation for $i=k+1$.  From the definition \eqref{eq:U-def}, the change of variables $r\mapsto -r$ gives the symmetry
\begin{equation}\label{eq:U-symmetry}
\mathcal U_{k,2k+2-i}(N)=\mathcal U_{k,i}(N).
\end{equation}
Applying \eqref{eq:L-internal} formally with $i=k+1$ and then using
\[
\mathcal U_{k,k+3}(N)=\mathcal U_{k,k-1}(N),
\]
which is the case $i=k-1$ of \eqref{eq:U-symmetry}, gives
\[
(1+q^{4N})\mathcal U_{k,k+1}(N)
=
2q^{2N}\mathcal U_{k,k-1}(N)
+
(q^{4N-2}+q^{2N-1})\mathcal U_{k,k+1}(N-1).
\]
Thus \eqref{eq:L-bottom} also holds with $L$ replaced by $\mathcal U$.

It remains to verify the two equations at the highest level.  Suppose first that $k$ is even.  Then
\[
\mathcal U_{k,1}(N)-q^{2N}\mathcal U_{k,3}(N)
=
q^{N^2}(-q;q^2)_N
\sum_{r=-N}^{N}(-1)^r
\frac{q^{kr^2+(k-2)r}(q^{2r}-q^{2N})}
{(q^2;q^2)_{N-r}(q^2;q^2)_{N+r}}.
\]
The term with $r=N$ is zero.  The remaining terms cancel in pairs under the involution
\[
r\longmapsto -r-1.
\]
Indeed, under this change of variables the summand changes sign, while the range
\[
-N\leq r\leq N-1
\]
is preserved.  Therefore
\[
\mathcal U_{k,1}(N)=q^{2N}\mathcal U_{k,3}(N),
\]
which is \eqref{eq:L-top-even} with $L$ replaced by $\mathcal U$.

Suppose next that $k$ is odd.  We consider
\[
\mathcal U_{k,2}(N)-q^{2N}\mathcal U_{k,4}(N)-q^{2N-1}\mathcal U_{k,2}(N-1).
\]
Using the ratio above to rewrite the $N-1$ layer in terms of the $N$ layer, this difference is equal to
\[
q^{N^2}(-q;q^2)_N
\sum_{r=-N}^{N}(-1)^r
\frac{
q^{kr^2+(k-1)r+2N-1}
(1-q^{2N-2r})(1+q^{2r+1})
}
{
(1+q^{2N-1})(q^2;q^2)_{N-r}(q^2;q^2)_{N+r}
}.
\]
The term with $r=N$ is zero because of the factor $1-q^{2N-2r}$.  On the remaining range, the summand is antisymmetric under the involution
\[
r\longmapsto -r-1.
\]
Hence the sum is zero.  Therefore
\[
\mathcal U_{k,2}(N)
=
q^{2N}\mathcal U_{k,4}(N)+q^{2N-1}\mathcal U_{k,2}(N-1),
\]
which is \eqref{eq:L-top-odd} with $L$ replaced by $\mathcal U$.

This proves that the functions $\mathcal U_{k,i}(N)$ satisfy the same recurrence system as the functions $L_{k,i}(N)$.
\qed

\begin{pro}\label{pro:uniqueness}
The recurrence system in Proposition \ref{pro:path-recurrence}, together with the initial values
$L_{k,i}(0)=1$
for all relevant indices $i$, has a unique solution over the ring of formal power series in $q$.
\end{pro}

\pf
Fix $N\geq1$ and arrange the current-layer unknowns $L_{k,i}(N)$, where
$i\equiv k+1\pmod 2$,
into a finite vector. Once the layer $N-1$ is known, the recurrences form a finite linear system for the layer $N$. Its diagonal entries are $1$ or $1+q^{4N}$, while every off-diagonal entry is divisible by $q$. Therefore, after reducing the coefficient matrix modulo $q$, we obtain the identity matrix. Hence the determinant of the coefficient matrix has constant term $1$, and so the matrix is invertible over the ring of formal power series in $q$.

Thus the layer $N$ is uniquely determined by the layer $N-1$. The result follows by induction from the initial layer $N=0$.
\qed

\begin{thm}\label{thm:path-product}
For $k\geq a\geq1$ with $k\not\equiv a\pmod 2$, the generating function for the lattice paths counted by $S_{k,a}(n)$ is given by
\begin{equation}\label{eq:path-product}
\sum_{n\geq0}S_{k,a}(n)q^n
=
\sum_{N\geq0}L_{k,a}(N)
=
\frac{(-q;q^2)_\infty
(q^a,q^{2k+2-a},q^{2k+2};q^{2k+2})_\infty}
{(q^2;q^2)_\infty}.
\end{equation}
\end{thm}

\pf
By Propositions \ref{pro:path-recurrence}, \ref{pro:U-recurrence}, and \ref{pro:uniqueness}, we have
\[
L_{k,i}(N)=\mathcal U_{k,i}(N)
\qquad
(N\geq0,\; i\equiv k+1\pmod2).
\]
Since $k\not\equiv a\pmod2$, we have $a\equiv k+1\pmod2$. Taking $i=a$ and summing over $N$ gives
\[
\sum_{N\geq0}L_{k,a}(N)
=
\sum_{N\geq0}\mathcal U_{k,a}(N).
\]
By \eqref{eq:product-side-expansion}, the last sum is equal to
\[
\frac{(-q;q^2)_\infty
(q^a,q^{2k+2-a},q^{2k+2};q^{2k+2})_\infty}
{(q^2;q^2)_\infty}.
\]
The first equality in \eqref{eq:path-product} follows from the definition of $L_{k,a}(N)$ as the fixed-peak generating function for the paths in Theorem \ref{main}. This proves \eqref{eq:path-product}.
\qed

Combining Theorem \ref{thm:path-product} with the lattice path interpretation in Theorem \ref{main} gives a recurrence-based proof of Theorem \ref{main1}. In the final section, we give an independent analytic proof of the same identity by using Bailey's lemma.

\section{Lattice path interpretation}

In this section, we provide a combinatorial interpretation of \eqref{s} in terms of lattice paths. 
We explain how the factors on the multisum side contribute to the construction of the corresponding lattice paths, verify that the resulting paths satisfy the required conditions, and finally show that the construction is reversible. 
This yields a bijection and hence proves Theorem \ref{main}.

To provide a lattice path interpretation of equation  \eqref{s}, some additional definitions and operations concerning lattice paths are required. 
 First, we define two fundamental operations: the volcanic uplift and the right-move operation.

\noindent\textbf{Volcanic uplift:} At each peak of the path, the lattice path is cut open, separated horizontally by two units, and a new peak is inserted whose height is one unit greater than the original peak (see Figure~\ref{volcanic_uplift}).

\begin{figure}[htbp]
	\centering
	\begin{tikzpicture}[scale=0.6, line width=1pt]
		
		\draw (0,0) -- (1,1) -- (2,0) -- (3,0);
		
		\draw[->, thick] (4, 0.5) -- (5, 0.5);
		
		\begin{scope}[shift={(6,0)}]
			\draw (0,0) -- (1,1);
			
			\draw (1,1) -- (2,2) -- (3,1);
			
			\draw (3,1) -- (4,0) -- (5,0);
			
			\draw[thick] (1, 1.10) -- (1, 0.90); 
			\draw[thick] (3, 1.10) -- (3, 0.90);
		\end{scope}
		
	\end{tikzpicture}
	\caption{The volcanic uplift operation.}
	\label{volcanic_uplift}
\end{figure}
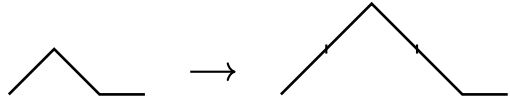

\noindent\textbf{Right-move:} We apply the right-move operation to a peak of relative height $1$. 
Depending on the local configuration, this operation falls into one of the following three cases, illustrated from top to bottom in Figure~\ref{right_move_operations}:

\begin{enumerate}
	\item The peak is followed by an East step $E$.
	\item The peak is followed by a $NE$ step, which is in turn followed immediately by another $NE$ step.
	\item The peak is preceded by the origin or by a $SE$ step, and is followed by a $SE$ step.
\end{enumerate}

These three cases are shown in Figure~\ref{right_move_operations}. 
A further exceptional case occurs when the weight difference between the current peak and the next peak to its right is exactly $2$; in this situation, the current peak is left unchanged and the move is transferred to the next peak, 
as illustrated in Figure~\ref{right_move_case4}. 
In each case, the relative height of the moving peak is either preserved at $1$ or exchanged with the relative height of its neighboring peak.
\begin{figure}[htbp]
	\centering
	\begin{tikzpicture}[scale=0.8, line width=1pt]
		
		\begin{scope}[shift={(0,0)}]
			\draw (0,0) -- (1,1) -- (2,0) -- (3,0);
			\draw[->, thick] (3.6,0.5) -- (4.6,0.5);
			\begin{scope}[shift={(5.6,0)}]
				\draw (0,0) -- (1,0) -- (2,1) -- (3,0);
			\end{scope}
			\node[font=\small] at (4.1,-0.9) {Case 1: move past an $E$ step};
		\end{scope}
		
		\begin{scope}[shift={(0,-3)}]
			\draw (0,0) -- (1,1) -- (2,0) -- (3,1);
			\draw[->, thick] (3.6,0.5) -- (4.6,0.5);
			\begin{scope}[shift={(5.6,0)}]
				\draw (0,0) -- (1,1) -- (2,2) -- (3,1);
				\draw[thick] (1,1.1) -- (1,0.9);
			\end{scope}
			\node[font=\small] at (4.1,-1.0) {Case 2: move past a $NE$ step};
		\end{scope}
		
		\begin{scope}[shift={(0,-7)}]
			\draw (0,1) -- (1,2) -- (2,1) -- (3,0);
			\draw[thick] (2,1.1) -- (2,0.9);
			\draw[->, thick] (3.6,0.5) -- (4.6,0.5);
			\begin{scope}[shift={(5.6,0)}]
				\draw (0,1) -- (1,0) -- (2,1) -- (3,0);
				
			\end{scope}
			\node[font=\small] at (4.1,-1.2) {Case 3: move past a $SE$ step};
		\end{scope}

	\end{tikzpicture}
	
	\caption{The three cases of the right-move operation.}
	\label{right_move_operations}
\end{figure}
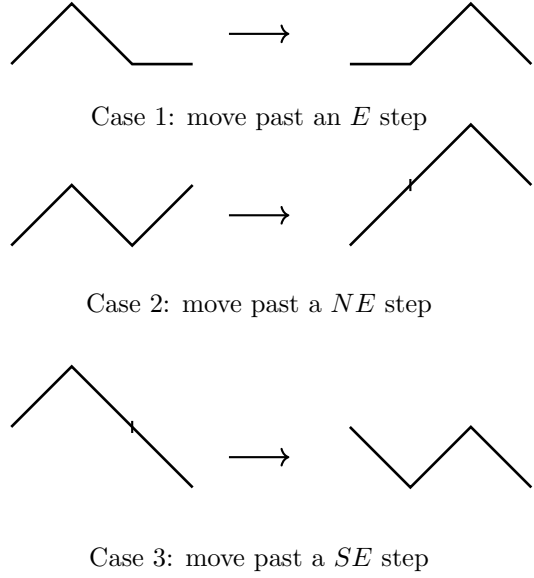

\begin{figure}[htbp]
	\centering
	\begin{tikzpicture}[scale=0.5, line width=1pt]
		
		\begin{scope}[shift={(0,0)}]
			\draw (0,0) -- (1,1) -- (2,0) -- (4,2) -- (5,1) -- (6,2) -- (8,0);
			
			\draw[thick] (3.08,1.08) -- (2.92,0.92);
			\draw[thick] (7.08,1.08) -- (6.92,0.92);
		\end{scope}
		
		\draw[->, thick] (8.8,1.2) -- (10.0,1.2);
		
		\begin{scope}[shift={(11,0)}]
			\draw (0,0) -- (2,2) -- (3,1) -- (4,2) -- (5,1) -- (6,2) -- (8,0);
			
			\draw[thick] (1.08,1.08) -- (0.92,0.92);
			\draw[thick] (7.08,1.08) -- (6.92,0.92);
		\end{scope}
		
		\draw[->, thick] (19.8,1.2) -- (21.0,1.2);
		
		\begin{scope}[shift={(22,0)}]
			\draw (0,0) -- (2,2) -- (3,1) -- (4,2) -- (6,0) -- (7,1) -- (8,0);
			
			\draw[thick] (1.08,1.08) -- (0.92,0.92);
			\draw[thick] (5.08,1.08) -- (4.92,0.92);
		\end{scope}
		
	\end{tikzpicture}
	\caption{If the current peak and the next peak to its right have weight difference $2$, then the current peak is left unchanged and the move is transferred to the next peak.}
	\label{right_move_case4}
\end{figure}
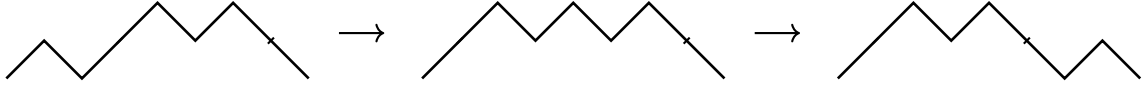




\medskip
\noindent\textbf{Proof of Theorem \ref{main}.}
\smallskip


We organize the construction according to the factors appearing on the multisum side of \eqref{s}. 
For each factor, we describe the corresponding lattice path operation and its contribution to the major index, and then show that the resulting construction is reversible.

\medskip
\noindent\textbf{\emph{Initial configuration.}}
We begin with the contribution corresponding to $N_{k-1}^{2}+2N_{k-1}$. 
First, insert $n_{k-1}$ peaks of relative height $1$ at the origin, with weights
\[
1,3,\dots,2n_{k-1}-1.
\]
Their total contribution to the major index is
\[
1+3+\cdots+(2n_{k-1}-1)=n_{k-1}^{2}=N_{k-1}^{2}.
\]
Next, insert two $SE$ steps at the beginning of the path. This shifts each of these peaks two units to the right, and hence contributes an additional $2N_{k-1}$ to the major index.

\medskip
\noindent\textbf{\emph{The factor $\boldsymbol{{1}/{(q^{4};q^{4})_{n_{k-1}}}}$.}}
Interpret this factor as the generating function for partitions
\[
4b_1\ge 4b_2\ge \cdots \ge 4b_{n_{k-1}}\ge 0.
\]
Starting from the rightmost peak of relative height $1$, insert $4b_1$ East steps before it, thereby moving it to the right by $4b_1$. 
Then insert $4b_2$ East steps before the next peak, and continue in the same way. 
This realizes the factor ${1}/{(q^{4};q^{4})_{n_{k-1}}}$.

\medskip
\noindent\textbf{\emph{The factor $\boldsymbol{(-q;q^{2})_{n_{k-1}}}$.}}
Applying a volcanic uplift to the $r$-th peak from the right increases the major index by $2r-1$. 
Therefore, the factor $(-q;q^{2})_{n_{k-1}}$ is realized by choosing some of the peaks of relative height $1$ created above and applying the volcanic uplift operation to them. 
As a consequence, the peaks of relative height $2$ produced in this way have even weights, while the remaining peaks of relative height $1$ have odd weights.

\medskip
\noindent\textbf{\emph{The inductive construction for $\boldsymbol{N_j^2}$, $\boldsymbol{1\le j\le k-2}$.}}
We proceed recursively for $j=k-2,k-3,\dots,1$. 
Assume that all peaks corresponding to $N_{j+1}$ have already been constructed. 
First, apply the volcanic uplift operation to each of these $N_{j+1}$ peaks, thereby increasing their relative heights by $1$. 
Then insert $n_j$ new peaks of relative height $1$ at the origin. 
The total contribution of this step to the major index is
\[
(1+3+\cdots+(2n_j-1))+2n_jN_{j+1}+N_{j+1}^{2}
= N_j^{2},
\]
since $N_j=n_j+N_{j+1}$.

\medskip
\noindent\textbf{\emph{The factor $\boldsymbol{{1}/{(q^{2};q^{2})_{n_j}}}$.}}
After the $n_j$ new peaks of relative height $1$ have been inserted, we apply the right-move operation to them in steps of size $2$. 
This yields the factor ${1}/{(q^{2};q^{2})_{n_j}}$.

\medskip
\noindent\textbf{\emph{The extra term $\boldsymbol{2N_j}$.}}
Whenever $j\ge a$ and $j\equiv a\equiv k-1\pmod 2$, we insert two additional $SE$ steps at the beginning of the path. 
This shifts all the currently existing $N_j$ peaks two units to the right, and hence contributes $2N_j$ to the major index.

\medskip
\noindent\textbf{\emph{Verification of the path conditions.}}
We now verify that the above construction produces a lattice path satisfying the conditions in Theorem \ref{main}. 
Indeed:
\begin{enumerate}
	\item Throughout the construction, pairs of $SE$ steps are inserted a total of ${(k+1-a)}/{2}$ times. Consequently, the starting point is shifted from $(0,0)$ to $(0,k+1-a)$.
	
	\item Every peak is created with relative height $1$ and can undergo at most $k-1$ volcanic uplifts. Therefore, the resulting path remains at or below height $k$.
	
	\item The $n_j$ peaks inserted at stage $j$ subsequently undergo exactly $j-1$ volcanic uplifts, and hence acquire final relative height $j$.
	
	\item The required parity condition is preserved throughout the construction. Indeed, the insertion of East steps, the insertion of pairs of $SE$ steps, and the right-move operation preserve the parity relation between the weight and the relative height of each existing peak, whereas the volcanic 
	uplift increases both quantities by $1$. Hence the weight and the relative height of every peak always have the same parity.
\end{enumerate}

Thus the resulting paths satisfy all the conditions stated in Theorem \ref{main}.

\medskip
\noindent\textbf{\emph{Reversibility of the construction.}}
It remains to explain how a lattice path counted by $S_{k,a}(n)$ can be read in reverse. Let $L$ be such a path. We recover the parameters by descending on the relative height. Fix $j$. After contracting all peaks of relative height less than $j$, the peaks of relative height $j$ remain unchanged, 
and the horizontal distances between them determine the effect of the right-move operation. Hence $n_j$ can be recovered from the resulting configuration.

When $j\ge a$, the forward construction also inserts pairs of initial $SE$ steps. Since the path starts at $(0,k+1-a)$, remains at or below height $k$, and satisfies the parity condition, after removing all peaks of relative height less than $j$, there are exactly
\[
2\left\lceil \frac{j-a+1}{2}\right\rceil
\]
initial $SE$ steps before the first peak of relative height $j$. These steps can therefore be removed uniquely. Thus the path can be reduced stage by stage, and at each stage we recover the corresponding contribution. This shows that the reverse procedure is well defined.

\medskip
\noindent\textbf{\emph{Conclusion.}}
Therefore, the above construction provides a combinatorial interpretation of the multisum side of \eqref{s}. Each summand on the left-hand side of \eqref{s} contributes lattice paths counted by $S_{k,a}(n)$, and the exponent of $q$ records the major index of the corresponding paths. Conversely, 
the reverse procedure shows how every such lattice path contributes to the left-hand side of \eqref{s}. Hence the left-hand side of \eqref{s} is the generating function for the lattice paths enumerated by $S_{k,a}(n)$. This establishes \eqref{s} combinatorially.

\qed

\begin{exam}\label{exam:52}
We illustrate the above construction in the case $(k,a)=(5,2)$ with
$(n_1,n_2,n_3,n_4)=(3,1,1,2).$
Then
$N_1=7, N_2=4, N_3=3, N_4=2,$
and hence
\[
N_1^2+N_2^2+N_3^2+N_4^2+2N_2+2N_4
=49+16+9+4+8+4=90.
\]

To specify a particular contribution, we choose $q^4$ 
from ${1}/{(q^4;q^4)_2}$, $q^3$ from $(-q;q^2)_2=(1+q)(1+q^3)$, $q^2$ 
from the factor ${1}/{(q^2;q^2)_1}$ corresponding to $n_3=1$, 
the constant term from the factor ${1}/{(q^2;q^2)_1}$ corresponding to
 $n_2=1$, and $q^6$ from ${1}/{(q^2;q^2)_3}$. The last contribution 
 is realized by moving the three new peaks by $4$, $2$, and $0$ units,
  respectively. Hence the total contribution is
\[
q^{90+4+3+2+0+6}=q^{105}.
\]

Based on this choice, we construct the corresponding lattice path.

\medskip
\noindent
Corresponding to $n_4=2$, we insert two peaks of relative height $1$ at the origin and then insert two initial $SE$ steps. 
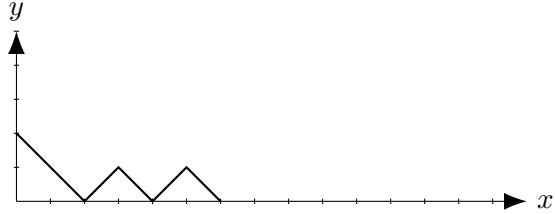
\begin{center}
\begin{tikzpicture}[scale=0.45]

    \draw[-{Latex[length=3mm]}] (0,0) -- (15,0) node[right] {$x$};
    \draw[-{Latex[length=3mm]}] (0,0) -- (0,5) node[above] {$y$};

    \foreach \x in {1,2,...,14}
        \draw (\x,0.08) -- (\x,-0.08);
  	\foreach \y in {1,2,...,5}
        \draw (0.08,\y) -- (-0.08,\y);
    \draw[thick]
        (0,2)
        -- (1,1)   
        -- (2,0)   
        -- (3,1)   
        -- (4,0)   
        -- (5,1)   
        -- (6,0);  

\end{tikzpicture}
\end{center}

Since
${1}/{(q^4;q^4)_2}$
contributes $q^4$, we move the rightmost of these two peaks four units to the right. Moreover, since
$(-q;q^2)_2=(1+q)(1+q^3)$
contributes $q^3$, we apply a volcanic uplift to the left peak. This gives the following path.

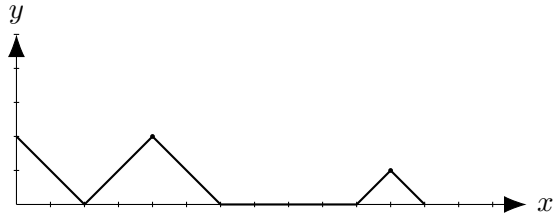
\begin{center}
\begin{tikzpicture}[scale=0.45]

    \draw[-{Latex[length=3mm]}] (0,0) -- (15,0) node[right] {$x$};
    \draw[-{Latex[length=3mm]}] (0,0) -- (0,5) node[above] {$y$};

    \foreach \x in {1,2,...,14}
        \draw (\x,0.08) -- (\x,-0.08);
  	\foreach \y in {1,2,...,5}
        \draw (0.08,\y) -- (-0.08,\y);
    \draw[thick]
        (0,2)
        -- (1,1)
        -- (2,0)
        -- (3,1)
        -- (4,2)
        -- (5,1)
        -- (6,0)
        -- (7,0) -- (8,0) -- (9,0) -- (10,0)
        -- (11,1)
        -- (12,0);

    \fill (4,2) circle (2pt);
    \fill (11,1) circle (2pt);

\end{tikzpicture}
\end{center}

\noindent
Corresponding to $n_3=1$, we apply volcanic uplift to the existing peaks and insert one new peak of relative height $1$ at the origin. 

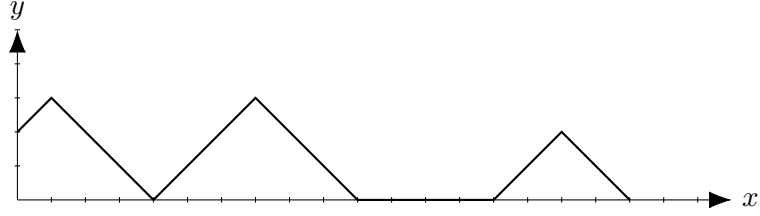
\begin{center}
\begin{tikzpicture}[scale=0.45]

    \draw[-{Latex[length=3mm]}] (0,0) -- (21,0) node[right] {$x$};
    \draw[-{Latex[length=3mm]}] (0,0) -- (0,5) node[above] {$y$};

    \foreach \x in {1,2,...,20}
        \draw (\x,0.08) -- (\x,-0.08);
  	\foreach \y in {1,2,...,5}
        \draw (0.08,\y) -- (-0.08,\y);
    \draw[thick]
        (0,2)
        -- (1,3)   
        -- (2,2)   
        -- (3,1)   
        -- (4,0)   
        -- (5,1)   
        -- (6,2)-- (7,3) -- (8,2)--(9,1)--(10,0)--(11,0)--(12,0)--(13,0)--(14,0)--(15,1)--(16,2)--(17,1)--(18,0);  

\end{tikzpicture}
\end{center}

Since${1}/{(q^2;q^2)_1}$
contributes $q^2$, we move this new peak two units to the right. We obtain the following lattice path.

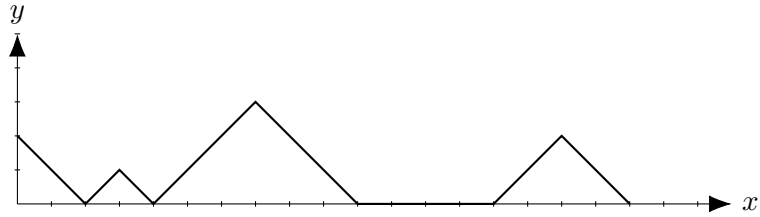
\begin{center}
\begin{tikzpicture}[scale=0.45]

    \draw[-{Latex[length=3mm]}] (0,0) -- (21,0) node[right] {$x$};
    \draw[-{Latex[length=3mm]}] (0,0) -- (0,5) node[above] {$y$};

    \foreach \x in {1,2,...,20}
        \draw (\x,0.08) -- (\x,-0.08);
  	\foreach \y in {1,2,...,5}
        \draw (0.08,\y) -- (-0.08,\y);
    \draw[thick]
        (0,2)
        -- (1,1)
        -- (2,0)
        -- (3,1)
        -- (4,0)
        -- (5,1)
        -- (6,2)
        -- (7,3)
        -- (8,2)
        -- (9,1)
        -- (10,0)
        -- (11,0) -- (12,0) -- (13,0) -- (14,0)
        -- (15,1)
        -- (16,2)
        -- (17,1)
        -- (18,0);

\end{tikzpicture}
\end{center}

\noindent
Corresponding to $n_2=1$, we again apply volcanic uplift to all existing peaks and insert one new peak of relative height $1$ at the origin. 
The factor ${1}/{(q^2;q^2)_1}$
corresponding to $n_2=1$ contributes the constant term, so no right-move is performed at this stage.

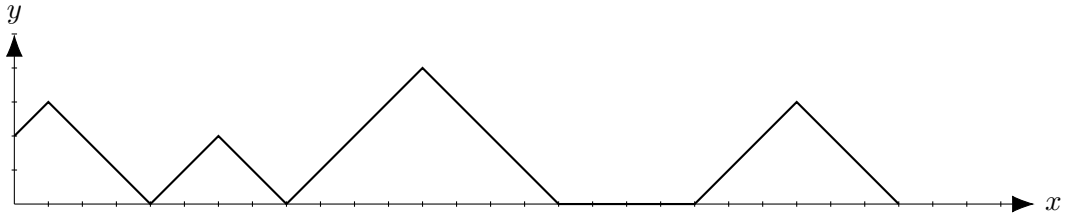
\begin{center}
\begin{tikzpicture}[scale=0.45]

    \draw[-{Latex[length=3mm]}] (0,0) -- (30,0) node[right] {$x$};
    \draw[-{Latex[length=3mm]}] (0,0) -- (0,5) node[above] {$y$};

    \foreach \x in {1,2,...,29}
        \draw (\x,0.08) -- (\x,-0.08);
  	\foreach \y in {1,2,...,5}
        \draw (0.08,\y) -- (-0.08,\y);
    \draw[thick]
        (0,2)
        -- (1,3)
        -- (2,2)
        -- (3,1)
        -- (4,0)
        -- (5,1)
        -- (6,2)
        -- (7,1)
        -- (8,0)
        -- (9,1)
        -- (10,2)
        -- (11,3) -- (12,4) -- (13,3) -- (14,2)
        -- (15,1)
        -- (16,0)
        -- (17,0)
        -- (18,0)-- (19,0) -- (20,0)--(21,1)--(22,2)--(23,3)--(24,2)--(25,1)--(26,0)
		;

\end{tikzpicture}
\end{center}

Since
$2\ge a$and $2\equiv a\equiv k-1\pmod 2,$
we insert two additional initial $SE$ steps. This yields the following path.

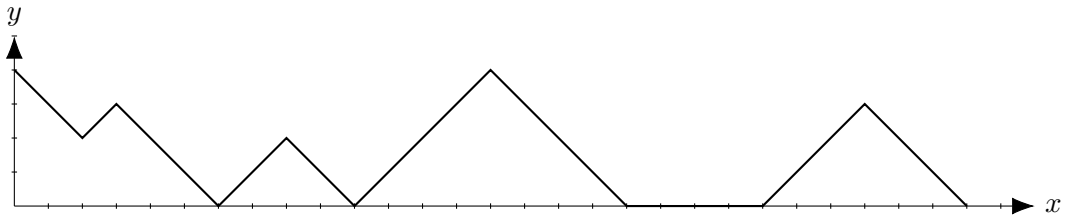
\begin{center}
\begin{tikzpicture}[scale=0.45]

    \draw[-{Latex[length=3mm]}] (0,0) -- (30,0) node[right] {$x$};
    \draw[-{Latex[length=3mm]}] (0,0) -- (0,5) node[above] {$y$};

    \foreach \x in {1,2,...,29}
        \draw (\x,0.08) -- (\x,-0.08);
  	\foreach \y in {1,2,...,5}
        \draw (0.08,\y) -- (-0.08,\y);
    \draw[thick]
		(0,4)--(1,3)
        --(2,2)
        -- (3,3)
        -- (4,2)
        -- (5,1)
        -- (6,0)
        -- (7,1)
        -- (8,2)
        -- (9,1)
        -- (10,0)
        -- (11,1)
        -- (12,2)
        -- (13,3) -- (14,4) -- (15,3) -- (16,2)
        -- (17,1)
        -- (18,0)
        -- (19,0)
        -- (20,0)-- (21,0) -- (22,0)--(23,1)--(24,2)--(25,3)--(26,2)--(27,1)--(28,0)
		;

\end{tikzpicture}
\end{center}

\noindent
Corresponding to $n_1=3$, we apply volcanic uplift to all existing peaks and insert three new peaks of relative height $1$ at the origin.

\begin{center}
\begin{tikzpicture}[scale=0.35]

    \draw[-{Latex[length=3mm]}] (0,0) -- (43,0) node[right] {$x$};
    \draw[-{Latex[length=3mm]}] (0,0) -- (0,5) node[above] {$y$};

    \foreach \x in {1,2,...,42}
        \draw (\x,0.08) -- (\x,-0.08);
  	\foreach \y in {1,2,...,5}
        \draw (0.08,\y) -- (-0.08,\y);
    \draw[thick]
		(0,4)--(1,5)--(2,4)--(3,5)--(4,4)--(5,5)--(6,4)--
		(7,3)
        --(8,2)
        -- (9,3)--(10,4)--(11,3)--(12,2)--(13,1) -- (14,0) -- (15,1) -- (16,2)
		--(17,3)--(18,2)--(19,1)--(20,0)--(21,1) -- (22,2) -- (23,3) -- (24,4) -- (25,5) -- (26,4) -- (27,3) -- (28,2)
		--(29,1)--(30,0)--(31,0)--(32,0)--(33,0)
		--(34,0)--(35,1)--(36,2)--(37,3)--(38,4)--(39,3)--(40,2)--(41,1)--(42,0)
      ;
\end{tikzpicture}
\end{center}

Finally, since
${1}/{(q^2;q^2)_3}$
contributes $q^6$, we move these three peaks by $4$, $2$, and $0$ units, respectively. Hence we arrive at the following final lattice path.

\begin{center}
\begin{tikzpicture}[scale=0.35]

    \draw[-{Latex[length=3mm]}] (0,0) -- (43,0) node[right] {$x$};
    \draw[-{Latex[length=3mm]}] (0,0) -- (0,5) node[above] {$y$};

    \foreach \x in {1,2,...,42}
        \draw (\x,0.08) -- (\x,-0.08);
  	\foreach \y in {1,2,...,5}
        \draw (0.08,\y) -- (-0.08,\y);
    \draw[thick]
		(0,4)--(1,5)--(2,4)--(3,3)--(4,2)--(5,3)--(6,2)--
		(7,3)
        --(8,4)
        -- (9,3)--(10,2)--(11,3)--(12,2)--(13,1) -- (14,0) -- (15,1) -- (16,2)
		--(17,3)--(18,2)--(19,1)--(20,0)--(21,1) -- (22,2) -- (23,3) -- (24,4) -- (25,5) -- (26,4) -- (27,3) -- (28,2)
		--(29,1)--(30,0)--(31,0)--(32,0)--(33,0)
		--(34,0)--(35,1)--(36,2)--(37,3)--(38,4)--(39,3)--(40,2)--(41,1)--(42,0)
      ;
\end{tikzpicture}
\end{center}

\noindent
Now we have constructed a lattice path $L$ with the following properties:
\begin{enumerate}
    \item $L$ starts at $(0,4)=(0,k+1-a)$ and remains at or below height $5$.
    
    \item The peaks have weights
    \[
    1,\ 5,\ 8,\ 11,\ 17,\ 25,\ 38,
    \]
    and hence the major index is
    \[
    1+5+8+11+17+25+38=105.
    \]
    This agrees with the exponent of $q$ in the chosen contribution.
    
    \item The relative heights of the peaks are
    \[
    1,\ 1,\ 2,\ 1,\ 3,\ 5,\ 4.
    \]
    
    \item The parity condition is satisfied: for every peak, its weight and relative height have the same parity.
    
    \item The number of $E$-steps between the starting point and each peak of relative height $5$ or $4$, and between any two such peaks, is a multiple of $4$.
\end{enumerate}
\end{exam}

The above example illustrates how the factors on the multisum side are realized by the lattice path operations.  Together with the fixed-peak recurrence argument in Section \ref{sec:fixedpeak}, this gives a combinatorial proof of Theorem \ref{main1}.  We now turn to an independent analytic proof based on Bailey pairs and Bailey's lemma.

\section{The Bailey pair}
In this section, we recall the notion of a Bailey pair 
and several consequences of Bailey's lemma that will be needed 
in the analytic proof of Theorem \ref{main1}. We begin with the 
definition of a Bailey pair and the unit Bailey pair, then state 
Bailey's lemma together with the special cases that will be used 
in the next section.

The notion of a Bailey pair arose from Bailey's proof of the Rogers--Ramanujan identities \cite{bai49}; 
see Andrews \cite{and85} for further details.

\begin{defi}\label{debailey}
 $(\alpha_n(a,q),\beta_n(a,q))$ is called a Bailey pair with parameters $(a,q)$ if
\[\beta_n(a,q)=\sum_{r=0}^n\frac{\alpha_r(a,q)}{(q;q)_{n-r}(aq;q)_{n+r}}\]
for all $n\geq 0$.
\end{defi}

All proofs in the next section start from the following unit Bailey pair.
\[
\beta_n(a,q)=\begin{cases}1,\ if\ n=0,\\0,\  if\ n>0,\end{cases}
\]

\[
\alpha_n(a,q)=(-1)^nq^{n\choose 2}\frac{(a;q)_n}{(q;q)_n}\frac{(1-aq^{2n})}{(1-a)}.
\]

Bailey's lemma, due to Bailey \cite{bai49} and reformulated by Andrews \cite{and84,and85}, produces a new Bailey pair from a given one.

\begin{thm}{\rm(Bailey's Lemma)}
Let  $(\alpha_n(a,q),\beta_n(a,q))$ form a Bailey pair with parameters $(a,q)$. $(\alpha'_n(a,q),\beta_n'(a,q))$ is also a  Bailey pair, where
\[
\alpha_n'(a,q)=
\frac{(\rho_1,\rho_2;q)_n}{(aq/\rho_1,aq/\rho_2;q)_n}
\left(\frac{aq}{\rho_1\rho_2}\right)^n
\alpha_n(a,q).
\]
and
\[
\beta_n'(a,q)=
\sum_{r=0}^n
\frac{(\rho_1,\rho_2;q)_r(aq/\rho_1\rho_2;q)_{n-r}}
{(aq/\rho_1,aq/\rho_2;q)_n(q;q)_{n-r}}
\left(\frac{aq}{\rho_1\rho_2}\right)^r
\beta_r(a,q).
\]
\end{thm}

A useful application of Bailey's lemma is the following result due to Paule \cite{pau85,pau87}. By letting $\rho_1,\rho_2 \to \infty$ in Bailey's lemma, we obtain the following limiting case, which is used in the proof of the Andrews--Gordon identity.

\begin{thm}\label{S1}
Let $(\alpha_n(a,q),\beta_n(a,q))$ be a Bailey pair with parameters $(a,q)$.
If \[\beta'_n(a,q)=\sum_{k=0}^n\frac{a^kq^{k^2}}{(q;q)_{n-k}}\beta_k(a,q),\]
then $(\alpha'_n(a,q),\beta'_n(a,q))$ is also a Bailey pair, where
\[\alpha'_r(a,q)=a^r q^{r^2}\alpha_r(a,q).\]
\end{thm}

The following three theorems are obtained by taking $\rho_1 \rightarrow \infty$, $\rho_2=-\sqrt{aq}$, $\rho_1\rightarrow \infty$,  $\rho_2=-a^{1/2}q$  and $\rho_1\rightarrow \infty$, $\rho_2=-a^{1/2}$
in Bailey's lemma respectively.

\begin{thm}\label{S2}
Suppose that $(\alpha_n(a,q),\beta_n(a,q))$ is a Bailey pair with parameters $(a,q)$.
If \[\beta'_n(a,q)=\sum_{k=0}^n\frac{(-\sqrt{aq};q)_k}
{(q;q)_{n-k}(-\sqrt{aq};q)_n}a^{k/2}q^{k^2/2}\beta_k(a,q),\]
then $(\alpha'_n(a,q),\beta'_n(a,q))$ is a Bailey pair with parameters $(a,q)$, where
\[\alpha'_r(a,q)=a^{r/2}q^{r^2/2}\alpha_r(a,q).\]
\end{thm}

The following theorem, due to Bressoud, Ismail and Stanton \cite{bre00}, changes the base from $q$ to $q^2$ in Bailey's lemma.

\begin{thm}\label{D1}
Let $(\alpha_n(a,q),\beta_n(a,q))$ be a Bailey pair with parameters $(a,q)$.
Then $(\alpha_n'(a,q),\beta_n'(a,q))$ is also a Bailey pair, where
\[\alpha_n'(a,q)=\alpha_n(a^2,q^2),\]
and
\[\beta_n'(a,q)=\sum_{r=0}^n\frac{(-aq;q)_{2r}}{(q^2;q^2)_{n-r}}
q^{n-r}\beta_r(a^2,q^2).\]
\end{thm}

We also need the following proposition, due to Bressoud, Ismail and 
Stanton \cite{bre00}, which allows us to introduce linear terms in the summation indices on the multisum side of our identities.

\begin{pro}{\rm (\cite[Proposition 4.1]{bre00})} \label{pro4.1}
Let $(\alpha_n(q),\beta_n(q))$ be a Bailey pair and
\[\alpha_n(q)=\begin{cases}1,&\mbox{for}\ n=0,\\(-1)^nq^{An^2}(q^{(A-1)n}+q^{-(A-1)n}),&\mbox{for}\ n>0,\end{cases}\]
then $(\alpha'_n(q),\beta'_n(q))$ is also a Bailey pair,
where $\beta'_n(q)=q^n\beta_n(q)$, and
\[\alpha'_n(q)=\begin{cases}1,&\mbox{for}\ n=0,\\(-1)^nq^{An^2}(q^{An}+q^{-An}),&\mbox{for}\ n>0.\end{cases}\]
\end{pro}

\section{The proof of Theorem \ref{main1}}
In this section, we prove Theorem \ref{main1} by using Bailey's lemma.

\noindent {\bf Proof of Theorem \ref{main1}.}
Let us begin with the unit Bailey pair by putting $a=1$ to get
 a Bailey pair $(\alpha_n^{(0)}(q),\beta_n^{(0)}(q))$, where
\[
\alpha_0^{(0)}(q)=1,\qquad
\alpha_n^{(0)}(q)=(-1)^nq^{n^2/2}(q^{-n/2}+q^{n/2}) \quad (n\ge1).
\]
\[\beta_n^{(0)}(q)=\begin{cases}1,\ if\ n=0, \\0,\  if\ n>0.\end{cases}\]

Applying Theorem \ref{D1} on $(\alpha_n^{(0)}(q),\beta_n^{(0)}(q))$, we get a new Bailey pair $(\alpha_n^{(1)}(q),\beta_n^{(1)}(q))$, where
\[\alpha_n^{(1)}(q)
=(-1)^nq^{n^2}(q^{-n}+q^{n}),\]
\[\beta_n^{(1)}(q)=\frac{q^n}{(q^2;q^2)_n}.\]

Applying Theorem \ref{S2} two times and  Proposition \ref{pro4.1} alternatively for $(k-a-1)/2$ times on $(\alpha_n^{(1)}(q),\beta_n^{(1)}(q))$ gives the following Bailey pair $(\alpha_n^{(k-a)}(q),\beta_n^{(k-a)}(q))$, where
\[
\alpha_n^{(k-a)}(q)
=(-1)^nq^{(k-a+1)n^2/2}(q^{-(k-a+1)n/2}+q^{(k-a+1)n/2}),
\]
\[
\beta_n^{(k-a)}(q)
=\sum_{n\geq N_1\geq \cdots \geq N_{k-a-1} \geq 0}
\frac{q^{n+N_1^2/2+\cdots+N_{k-a-1}^2/2+N_2+N_4+\cdots+N_{k-a-1}}(-\sqrt{q};q)_{N_{k-a-1}}}
{(q)_{n-N_1}\cdots(q)_{N_{k-a-2}-N_{k-a-1}}(q^2;q^2)_{N_{k-a-1}}(-\sqrt{q};q)_{n}}.
\]

We next apply Theorem \ref{S2} for $a$ times on this Bailey pair to obtain  $(\alpha_n^{(k)}(q),\beta_n^{(k)}(q))$.

\[
\alpha_n^{(k)}(q)
=(-1)^nq^{(k+1)n^2/2}(q^{-(k-a+1)n/2}+q^{(k-a+1)n/2}),
\]
\[
\beta_n^{(k)}(q)
=
\sum_{n\geq N_1\geq \cdots \geq N_{k-1}\geq 0}
\frac{
q^{N_1^2/2+\cdots+N_{k-1}^2/2+N_a+N_{a+2}+\cdots+N_{k-1}}
(-\sqrt q;q)_{N_{k-1}}
}
{
(q)_{n-N_1}\cdots(q)_{N_{k-2}-N_{k-1}}
(q^2;q^2)_{N_{k-1}}
(-\sqrt q;q)_n
}.
\]

It is easy to get  the following identity according to Definition \ref{debailey}.
\begin{align*}
\sum_{n\geq N_1\geq \cdots \geq N_{k-1}\geq 0}
&\frac{
q^{N_1^2/2+\cdots+N_{k-1}^2/2+N_a+N_{a+2}+\cdots+N_{k-1}}
(-\sqrt q;q)_{N_{k-1}}
}
{
(q)_{n-N_1}\cdots(q)_{N_{k-2}-N_{k-1}}
(q^2;q^2)_{N_{k-1}}
(-\sqrt q;q)_n
}  \\
&=
\sum_{r=0}^{n}
\frac{
(-1)^r q^{(k+1)r^2/2}
\left(q^{-(k-a+1)r/2}+q^{(k-a+1)r/2}\right)
}
{(q;q)_{n-r}(q;q)_{n+r}}.
\end{align*}

Letting $n\rightarrow \infty$, we see that
\begin{align}\label{eq1}
\sum_{ N_1\geq \cdots \geq N_{k-1} \geq 0}&
\frac{q^{N_1^2/2+\cdots+N_{k-1}^2/2+N_a+N_{a+2}+\cdots+N_{k-1}}(-\sqrt{q};q)_{N_{k-1}}}
{(q)_{N_1-N_2}\cdots(q)_{N_{k-2}-N_{k-1}}(q^2;q^2)_{N_{k-1}}}\nonumber\\
=&\frac{(-\sqrt{q};q)_{\infty}}{(q;q)_{\infty}}
\sum_{r=0}^{\infty}(-1)^rq^{(k+1)r^2/2}(q^{-(k-a+1)r/2}+q^{(k-a+1)r/2}).
\end{align}

By substituting $q^2$ for $q$ in \eqref{eq1},  Theorem \ref{main1} follows from Jacobi's triple product identity immediately.   \qed

\medskip
\begin{rem}\label{rem:bgg-dictionary}
There is a simple product-side coincidence with the Bressoud--G\"ollnitz--Gordon
dictionary of Layne, Marshall, Sadowski and Shambaugh \cite{layne23} in the
subcase where $k$ is even and $a$ is odd.  Write
\[
k=2K-2,\qquad a=2h-1.
\]
Then $i=K-h+1$ gives
\[
2K-2i+1=a,\qquad 2K+2i-3=2k+2-a.
\]
Hence the product side of Theorem~\ref{main1} is exactly the shelf-zero series
$G_{K-h+1}$ in \cite{layne23}.  By \cite[Proposition~7.1]{layne23}, this gives
\[
\frac{(-q;q^2)_\infty
(q^a,q^{2k+2-a},q^{2k+2};q^{2k+2})_\infty}
{(q^2;q^2)_\infty}
=
\widetilde J_{\frac{k+2}{2},\,\frac{a+1}{2}}
\left(\frac1q;1;q^2\right).
\]
This is only a product-side identification.  The other opposite-parity subcase,
where $k$ is odd and $a$ is even, would force $K=(k+2)/2$ to be a half-integer
and therefore does not fall within the integer-parameter Bressoud--G\"ollnitz--Gordon
shelf system of \cite{layne23}.  It would be interesting to know whether this
missing half admits an analogous dictionary, perhaps after modifying the BGG-type
shelf system or passing to a different family of ghost series.
\end{rem}

\begin{rem}
Although Andrews' identity already gives the product form in the case
$k\equiv a \pmod{2}$, it would be natural to ask whether there exists a
companion multisum identity in this parity case whose left-hand side admits
the same lattice path interpretation as in the present paper. Finding such
an identity would provide a more unified picture for the two parity cases,
and may also shed further light on the relation with ghost series
\cite{layne23} and possible $n$-color or split $(n+t)$-color partition
models in the spirit of modified lattice path interpretations
\cite{sacag23}.
\end{rem}

For completeness, we also give a direct proof of \eqref{eq:finite-transform}
from the $q$-Gauss summation.

\medskip
\noindent{\it A direct proof of \eqref{eq:finite-transform}.}
The summand is invariant under $r\mapsto -r$, so it is enough to prove the result
for $r\geq0$. Put $N=r+m$. The left-hand side of \eqref{eq:finite-transform},
divided by $q^{r^2}$, becomes
\[
\frac{(-q;q^2)_r}{(q^2;q^2)_{2r}}
\sum_{m\geq0}
\frac{(-q^{2r+1};q^2)_m q^{m^2+2rm}}
{(q^2;q^2)_m(q^{4r+2};q^2)_m}.
\]
We use the $q$-Gauss summation with base $q^2$,
\[
\sum_{m\geq0}
\frac{(A,B;q^2)_m}{(q^2,C;q^2)_m}
\left(\frac{C}{AB}\right)^m
=
\frac{(C/A,C/B;q^2)_\infty}{(C,C/(AB);q^2)_\infty}.
\]
Taking
\[
A=-q^{2r+1},\qquad C=q^{4r+2},
\]
and then letting $B\to\infty$, we obtain
\[
(B;q^2)_m\left(-\frac{q^{2r+1}}{B}\right)^m
\longrightarrow q^{m^2+2rm}.
\]
Hence
\[
\sum_{m\geq0}
\frac{(-q^{2r+1};q^2)_m q^{m^2+2rm}}
{(q^2;q^2)_m(q^{4r+2};q^2)_m}
=
\frac{(-q^{2r+1};q^2)_\infty}{(q^{4r+2};q^2)_\infty}.
\]
Therefore
\[
\frac{(-q;q^2)_r}{(q^2;q^2)_{2r}}
\frac{(-q^{2r+1};q^2)_\infty}{(q^{4r+2};q^2)_\infty}
=
\frac{(-q;q^2)_\infty}{(q^2;q^2)_\infty}.
\]
Multiplying back by $q^{r^2}$ proves \eqref{eq:finite-transform}. \qed

\vspace{0.5cm}
\noindent{\bf Acknowledgments.}
The authors would like to thank S. Ole Warnaar for valuable correspondence
concerning the historical development of lattice-path and particle-motion
methods for Andrews--Gordon type identities.  In particular, we are grateful
to him for drawing our attention to his work \cite{war97} and to the recent
work of Dousse and Jang \cite{doussejang26}.  His comments helped us improve
the historical discussion and clarify the relation between the present paper
and earlier lattice-path and particle-motion approaches.

This work was supported by the National Natural Science Foundation of China
(No.~12101307) and the Qing Lan Project of Jiangsu Province.


\begin{thebibliography}{99}\small

\bibitem{and74}
G. E. Andrews,
An analytic generalization of the Rogers--Ramanujan identities for odd moduli,
\emph{Proc. Nat. Acad. Sci. USA} \textbf{71} (1974), 4082--4085.

\bibitem{and84}
G. E. Andrews,
Multiple series Rogers--Ramanujan type identities,
\emph{Pacific J. Math.} \textbf{114} (1984), 267--283.

\bibitem{and85}
G. E. Andrews,
\emph{$q$-Series: Their Development and Application in Analysis, Number Theory, Combinatorics, Physics, and Computer Algebra},
CBMS Regional Conference Series in Mathematics, vol.~66,
American Mathematical Society, Providence, RI, 1985.

\bibitem{and10}
G. E. Andrews,
Parity in partition identities,
\emph{Ramanujan J.} \textbf{23} (2010), 45--90.

\bibitem{andbre85}
G. E. Andrews and D. M. Bressoud,
On the Burge correspondence between partitions and binary words,
\emph{Rocky Mountain J. Math.} \textbf{15} (1985), 225--233.

\bibitem{bai49}
W. N. Bailey,
Identities of the Rogers--Ramanujan type,
\emph{Proc. London Math. Soc.} \textbf{50} (1949), 1--10.

\bibitem{berpau01}
A. Berkovich and P. Paule,
Lattice paths, $q$-multinomials and two variants of the Andrews--Gordon identities,
\emph{Ramanujan J.} \textbf{5} (2001), 409--425.

\bibitem{bre89}
D. M. Bressoud,
Lattice paths and the Rogers--Ramanujan identities,
in \emph{Number Theory, Madras 1987},
Lecture Notes in Mathematics, vol.~1395,
Springer, Berlin, 1989, 140--172.

\bibitem{bre00}
D. M. Bressoud, M. Ismail, and D. Stanton,
Change of base in Bailey pairs,
\emph{Ramanujan J.} \textbf{4} (2000), 435--453.

\bibitem{bur82}
W. H. Burge,
A three-way correspondence between partitions,
\emph{European J. Combin.} \textbf{3} (1982), 195--213.

\bibitem{doussejang26}
J. Dousse and J. Jang,
Andrews--Gordon type identities with parity restrictions through particle motion,
preprint, arXiv:2603.05300, 2026.

\bibitem{gr04}
G. Gasper and M. Rahman,
\emph{Basic Hypergeometric Series}, second ed.,
Encyclopedia of Mathematics and its Applications, vol.~96,
Cambridge University Press, Cambridge, 2004.

\bibitem{gor61}
B. Gordon,
A combinatorial generalization of the Rogers--Ramanujan identities,
\emph{Amer. J. Math.} \textbf{83} (1961), 393--399.

\bibitem{kim13}
S. Kim and A. J. Yee,
The Rogers--Ramanujan--Gordon identities, the generalized G{\"o}llnitz--Gordon identities, and parity questions,
\emph{J. Combin. Theory Ser. A} \textbf{120} (2013), 1038--1056.

\bibitem{kur09}
K. Kur\c{s}ung\"{o}z,
Parity considerations in Andrews--Gordon identities,
\emph{European J. Combin.} \textbf{31} (2010), 976--1000.

\bibitem{layne23}
J. Layne, S. Marshall, C. Sadowski, and E. Shambaugh,
Ghost series and a motivated proof of the Bressoud--G{\"o}llnitz--Gordon identities,
\emph{Ramanujan J.} \textbf{65} (2024), 593--635.

\bibitem{mallet08}
O. Mallet,
\emph{Autour des surpartitions et des identit\'es de type Rogers--Ramanujan},
Ph.D. thesis, Universit\'e Paris Diderot--Paris 7, 2008.
Available at HAL, tel-00366067.

\bibitem{pau85}
P. Paule,
On identities of the Rogers--Ramanujan type,
\emph{J. Math. Anal. Appl.} \textbf{107} (1985), 255--284.

\bibitem{pau87}
P. Paule,
A note on Bailey's lemma,
\emph{J. Combin. Theory Ser. A} \textbf{44} (1987), 164--167.

\bibitem{sacag23}
R. Sachdeva and A. K. Agarwal,
Further Rogers--Ramanujan type identities for modified lattice paths,
\emph{Contrib. Discrete Math.} \textbf{18} (2023), 74--90.

\bibitem{war97}
S. O. Warnaar,
The Andrews--Gordon identities and $q$-multinomial coefficients,
\emph{Commun. Math. Phys.} \textbf{184} (1997), 203--232.


\end{thebibliography}
\end{document}